\documentclass[a4paper,10pt,reqno, english]{amsart}
\usepackage[utf8]{inputenc}
\usepackage[T1]{fontenc}
\usepackage{amsmath,amsthm}
\usepackage{amsfonts,amssymb,enumerate}
\usepackage{url,paralist}
\usepackage{mathtools,color}
\usepackage[colorlinks=true,urlcolor=blue,linkcolor=red,citecolor=magenta]{hyperref}
\usepackage{enumerate}
\usepackage{anysize}  
\usepackage{tikz-cd}
\usepackage{bm}

\theoremstyle{plain}
\newtheorem{theorem}{Theorem}[section]
\newtheorem{lemma}[theorem]{Lemma}

\theoremstyle{definition}
\newtheorem{definition}[theorem]{Definition}

\newcommand{\R}{\mathbb{R}}

\newcommand{\A}{\mathcal{A}}

\newcommand{\gr}{\operatorname{G}}
\newcommand{\GL}{\operatorname{GL}}
\newcommand{\Mat}{\operatorname{Mat}}
\renewcommand{\span}{\operatorname{span}}
\newcommand{\GP}[1]{\operatorname{P}_Z(#1)}
\newcommand{\nng}[2]{\operatorname{G}_{#1}^{\ge0}(\R^{#2})}
\newcommand{\GPZ}[2]{\operatorname{P}_{#2}(#1)}
\newcommand{\diag}{\operatorname{diag}}
\newcommand{\id}{\operatorname{id}}

\newcommand\RP{\mathbb{R}{\mathrm{P}}}

\begin{document}

\title{Some more amplituhedra are contractible}


\author[Blagojevi\'c]{Pavle V. M. Blagojevi\'{c}}
\address{Institute of Mathematics, Freie Universität Berlin, Arnimallee 2, 14195 Berlin, Germany\hfill\break%
\mbox{\hspace{4mm}}Mat.\ Institut SANU, Knez Mihailova 36, 11000 Beograd, Serbia}
\email{blagojevic@math.fu-berlin.de}

\author[Galashin]{Pavel Galashin}
\address{Department of Mathematics, Massachusetts Institute of Technology, 77 Massachusetts Avenue,\hfill\break%
\mbox{\hspace{4mm}}Cambridge, MA 02139, USA}
\email{galashin@mit.edu}

\author[Pali\'c]{Nevena Pali\'c}
\address{Institute of Mathematics, Freie Universität Berlin, Arnimallee 2, 14195 Berlin, Germany}
\email{palic@math.fu-berlin.de}

\author[Ziegler]{G\"unter M. Ziegler} 
\address{Institute of Mathematics, Freie Universität Berlin, Arnimallee 2, 14195 Berlin, Germany}
\email{ziegler@math.fu-berlin.de}

\thanks{The research by Pavle V. M. Blagojevi\'{c} has received funding from the grant ON 174024 of the Serbian Ministry of Education and Science.\newline\indent
The research by Nevena Pali\'c has received funding from DFG via the Berlin Mathematical School.\newline\indent
This material is based on work supported by the National Science Foundation under Grant No.\ DMS-1440140 during the Fall of 2017, while all authors were in residence at the Mathematical Sciences Research Institute in Berkeley CA}
\date{\today}

\begin{abstract} 
The amplituhedra arise as images of the totally nonnegative Grassmannians 
by projections that are induced by linear maps. 
They were introduced in Physics by Arkani-Hamed \& Trnka (Journal of High Energy Physics, 2014) as model spaces that should provide a better understanding of the scattering amplitudes of quantum field theories. The topology of the amplituhedra has been known only in a few special cases, where they turned out to be homeomorphic to balls. The amplituhedra are special cases of Grassmann polytopes introduced by Lam (Current Developments in Mathematics 2014, Int.\ Press).
In this paper we show that that some further amplituhedra are homeomorphic to balls, and that some more Grassmann polytopes and amplituhedra are contractible. 
\end{abstract}

\maketitle


\section{Introduction and statement of the main result}
\label{sec : Introduction and the statement of the main results }

\subsection{Introduction}
Let $n$ and $k$ be integers such that $n \geq k \geq 1$. If $\Mat_{k,n}$ denotes the space of all real $k\times n$ matrices of rank $k$,
then the real Grassmannian $\gr_k(\R^n)$ --- the space of all $k$-dimensional linear subspaces of $\R^n$ ---
can be defined as the orbit space $\gr_k(\R^n) = \GL_k \backslash \Mat_{k,n}$. 
The totally nonnegative part of the Grassmannian is defined quite analogously.

\begin{definition}[Postnikov {\cite[Sec.\,3]{Postnikov2006}}]
Let $n \geq k \geq 1$ be integers, let $\Mat_{k,n}^{\ge0}$ be the space of all real $k \times n$ matrices of rank $k$ all whose  maximal minors are nonnegative, and let $\GL_k^+$ denote the group of all real $k\times k$ matrices with positive determinant, which acts freely on $\Mat_{k,n}^{\ge0}$ by matrix multiplication from the left.
The \emph{totally nonnegative Grassmannian} $\gr_k^{\ge0}(\R^n)$ is the orbit space $\gr_k^{\ge0}(\R^n) = \GL_k^+ \backslash \Mat_{k,n}^{\ge0}$.
\end{definition}

The totally nonnegative Grassmannian was introduced and studied by Postnikov in 2006 \cite[Sec.\,3]{Postnikov2006},
building on works by Lusztig \cite{Lusztig1994} and by Fomin \& Zelevinsky \cite{Fomin1999}.
Subsequently, the geometric and combinatorial properties of the totally nonnegative Grassmannian were studied intensively. 
Rietsch \& Williams showed that the totally nonnegative Grassmannian is contractible \cite[Thm.\,1.1]{Rietsch2010}; an earlier argument by Lusztig \cite[Sec.\,4.4]{Lusztig1998} can also be adapted to prove the same. Galashin, Karp \& Lam \cite[Thm.\,1.1]{Galashin2017} proved that $\gr_k^{\ge0}(\R^n)$ is indeed homeomorphic to a closed $k(n-k)$-dimensional ball. 

In 2014, the physicists Arkani-Hamed \& Trnka \cite[Sec.\,9]{Arkani-Hamed2014} introduced the amplituhedra as certain images of the totally nonnegative Grassmannians. 
They conjectured that their geometry describes scattering amplitudes in some quantum field theories.
For a gentle introduction to amplituhedra in physics and mathematics consult \cite{Bourjaily2018}.
Shortly after, Lam introduced Grassmann polytopes \cite{Lam2016}, which generalize amplituhedra.

Postnikov \cite[Def.\,3.2, Thm.\,3.5]{Postnikov2006} defined a CW structure on the totally nonnegative Grassmannian $\gr_k^{\ge0}(\R^n)$ such that each cell, also called a \emph{positroid cell}, is indexed by the associated matroid -- a positroid -- of rank $k$ on $n$ elements, see also \cite{Postnikov2009}. 
Furthermore, Rietsch \& Williams \cite{Rietsch2010} showed that the closures of positroid cells are contractible and that their boundaries are homotopy equivalent to spheres.

\begin{definition}
\label{def:amplituhedron}
Let $k\geq 1$, $m\ge0$ and $n \geq k+m$ be integers, and let $Z$ be a real $(k+m)\times n$ matrix such that the assignment
\begin{equation}
\label{eqn:definition map Z}
\widetilde{Z}(\span(V))=\span(VZ^\top)
\end{equation}
induces a map
\[
\widetilde{Z} \colon \gr_k^{\ge0}(\R^n) \longrightarrow \gr_k(\R^{k+m}).
\]
Here $V\in \Mat_{k,n}^{\geq 0}$, $\span$ denotes the row span of a matrix, and $Z^\top$ is the transpose of the matrix $Z$.
The image $\widetilde{Z}(\bar{e})$ of a closed positroid cell $\bar{e}$ in the CW decomposition of the nonnegative Grassmannian $\gr_k^{\ge0}(\R^n)$ is called a \emph{Grassmann polytope}, denoted by $\GP{e}$. 
If $e$ is the maximal cell, which for this CW decomposition means  $\bar{e} = \gr_k^{\ge0}(\R^n)$, and all $(k+m)\times (k+m)$ minors of the matrix $Z$ are positive, then the Grassmann polytope $\GP{e}$ is called an \emph{amplituhedron} and is denoted by $\A_{n,k,m}(Z)$.
\end{definition}

The previous definition in particular means that if $v_1, \dots, v_k \in \R^n$ are linearly independent row vectors, then 
\[
\widetilde{Z}(\span\{v_1, \dots, v_k\}) = \span\{v_1Z^\top, \dots, v_kZ^\top\}.
\]
The map $\widetilde{Z}$ is said to be \emph{well defined} if $\span(VZ^\top)$ is a $k$-dimensional subspace of $\R^{k+m}$ for every $V \in \Mat_{k,n}^{\ge0}$.
The fact that the map $\widetilde{Z}$ is well defined when $Z$ is a matrix with positive maximal minors was established by Arkani-Hamed \& Trnka in \cite{Arkani-Hamed2014} and by Karp in \cite[Thm.\,4.2]{Karp2017b}. Lam \cite[Prop.\,15.2]{Lam2016}, however, considers a larger class of matrices $Z$ for which the map $\widetilde{Z}$ is still well defined.

The structure of the amplituhedron is known only in a few cases. 
In the case $m=0$ all amplituhedra $\A_{n,k,0}(Z)$ are the point $\gr_k(\R^{k})$, whereas when $m=1$ Karp \& Williams \cite[Cor.\,6.18]{Karp2017} have shown that the amplituhedron is homeomorphic to a ball. 
For $k=1$ the amplituhedron is a cyclic polytope of dimension $m$ on $n$ vertices \cite{Sturmfels1988}, and for $n=k+m$ the map $Z$ is a linear isomorphism, and consequently the amplituhedron is homeomorphic to the totally nonnegative Grassmannian $\gr_k^{\ge0}(\R^n)$, which is a ball by \cite[Thm.\,1.1]{Galashin2017}. Finally, Galashin, Karp \& Lam \cite[Thm.\,1.2]{Galashin2017} proved that the cyclically symmetric amplituhedra, amplituhedra arising from particularly chosen matrices $Z$, are homeomorphic to balls whenever $m$ is even. 
The topology of other Grassmann polytopes is unknown.

\medskip

\subsection{Main results}

Our first result gives a family of contractible Grassmann polytopes.

\begin{theorem}
\label{thm:contractible}
Let $k \geq 1$ and $m \ge0$ be integers, and let $Z$ be a real $(k+m)\times (k+m+1)$ matrix such that the map $\widetilde{Z}\colon \nng{k}{k+m+1} \longrightarrow \gr_k(\R^{k+m})$ is well defined. 
Then the Grassmann polytope $\GP{e}$ is contractible for every positroid cell $e$ in the CW decomposition of $\gr_k^{\ge0}(\R^{k+m+1})$.
\end{theorem}

The proof of Theorem \ref{thm:contractible} relies on classical results of Smale \cite[Main Thm.]{Smale1957} and Whitehead \cite[Thm.\,1]{Whitehead1949} combined with the fact that every Grassmann polytope admits a triangulation (as a topological space), see Theorem \ref{thm:triangulation}.

\medskip

The following is a consequence of Smale's result \cite[Main Thm.]{Smale1957}.

\begin{theorem}[Smale] 
\label{thm:Smale}
Let $X$ and $Y$ be path connected, locally compact, separable metric spaces, and in addition let $X$ be locally contractible.
	Let $f\colon X\longrightarrow Y$ be a continuous surjective proper map, that is, any inverse image of a compact set is compact.
	If for every $y\in Y$ the inverse image $f^{-1}(\{y\})$ is contractible, then the induced homomorphism
	\[
	f_{\#} \colon  \pi_i(X)\longrightarrow \pi_i(Y)
	\]
	is an isomorphism for all $i \geq 0$.
\end{theorem}

Recall that a continuous map $f\colon X\longrightarrow Y$ between topological spaces $X$ and $Y$ is a {\em weak homotopy equivalence} if the induced map on the path connected components $f_{\#}\colon\pi_0(X)\longrightarrow\pi_0(Y)$ is bijective, and for every point $x_0\in X$ and for every integer $n \geq 1$ the induced map $f_{\#}\colon\pi_n(X,x_0)\longrightarrow\pi_n(Y,f(x_0))$ is an isomorphism.

\begin{theorem}[{\cite[Thm.\,1]{Whitehead1949}}]
\label{thm:Whitehead}
Let $X$ and $Y$ be topological spaces that are homotopy equivalent to CW complexes. Then a continuous map $f\colon X \longrightarrow Y$ is a weak homotopy equivalence if and only if it is a homotopy equivalence.
\end{theorem}

Since Theorem \ref{thm:Whitehead} requires that spaces have the homotopy type of a CW complex, the following theorem is a necessary ingredient in the proof of Theorem \ref{thm:contractible}.

\begin{theorem}
\label{thm:triangulation}
Every Grassmann polytope is semi-algebraic as a subset of a Grassmannian.
In particular, every Grassmann polytope is homeomorphic to a semi-algebraic subset of some real affine space, and admits a triangulation.
\end{theorem} 

Note that Theorem \ref{thm:triangulation} claims that every Grassmann polytope $\GP{e}$ can be triangulated in a classical sense, thus there exists a simplicial complex $T$ and a homeomorphism $T \longrightarrow \GP{e}$. 
A very similar argument to ours was also given by Arkani-Hamed, Bai \& Lam in \cite[Appendix~J]{Arkani-Hamed2017}. 

\medskip
In order to apply Theorem \ref{thm:Smale} to the map $\widetilde{Z}$,
we need to understand its fibers. Thus we prove the following theorem.

\begin{theorem}
\label{thm:fibers}
Let $k \geq 1$ and $m \ge0$ be integers, and let $Z$ be a real $(k+m)\times (k+m+1)$ matrix such that the map $\widetilde{Z}$ is well defined. 
Then for every positroid cell $e$ and for every point $y \in \GP{e}$, the inverse image $(\widetilde{Z}|_{\bar{e}})^{-1}(\{y\}) = \widetilde{Z}^{-1}(\{y\}) \cap \bar{e}$ under the restriction map $\widetilde{Z} |_{\bar{e}} \colon \bar{e} \longrightarrow \GP{e}$ is contractible.
\end{theorem}

The proof of Theorem \ref{thm:fibers} is postponed to the next section.
Here we show that Theorem \ref{thm:fibers} in combination with Theorems \ref{thm:Smale}--\ref{thm:triangulation} implies our main result.

\begin{proof}[Proof of Theorem \ref{thm:contractible}]
Let $e$ be a positroid cell in the CW decomposition of $\nng{k}{k+m+1}$. We apply Theorem \ref{thm:Smale} to the map $\widetilde{Z}: \bar{e} \longrightarrow \GP{e}$.
The spaces $\bar{e}$ and $\GP{e}$, as well as the map $\widetilde{Z}$, satisfy the assumptions of Theorem \ref{thm:Smale}.
Furthermore, Theorem \ref{thm:fibers} implies that for every $y \in \bar{e}$, the fiber $\widetilde{Z}^{-1}(\{y\})$ is contractible.
Thus, from Theorem \ref{thm:Smale} we have that the map $\widetilde{Z}$ is a weak homotopy equivalence. 

The closed positroid cell $\bar{e}$ is a CW complex.
Furthermore, the Grassmann polytope $\GP{e}$ is a CW complex, by Theorem \ref{thm:triangulation}.
Thus, from Theorem \ref{thm:Whitehead}, we conclude that the map $\widetilde{Z}$ is a homotopy equivalence. 
Hence, the Grassmann polytope $\GP{e}$ is homotopy equivalent to the closed positroid cell $\bar{e}$, which is contractible, see \cite[Thm.\,1.1]{Rietsch2010}. 
\end{proof}

Theorem \ref{thm:contractible} in particular implies that all amplituhedra $\A_{k+m+1,k,m}(Z)$ are contractible.
Our next result shows that if in addition $m$ is even, they are homeomorphic to balls. 

\begin{theorem}
\label{thm:homeomorphic}
Let $k \geq 1$ be an integer, let $m \ge0$ be an even integer, and let $Z \in \Mat_{k+m,k+m+1}$ be a matrix with all $(k+m)\times (k+m)$ minors positive. 
Then the amplituhedron $\A_{k+m+1,k,m}(Z)$ induced by the matrix $Z$ is homeomorphic to a $km$-dimensional ball.
\end{theorem}

The proof of Theorem \ref{thm:homeomorphic} is presented in Section \ref{sec:homeomorphic}.  We remark that the combinatorics of the amplituhedron in the case $n=k+m+1$ with $m$ even is identical to that of a cyclic polytope, see \cite{Galashin2018}. 

\subsection*{Acknowledgement} 
The authors thank Rainer Sinn for sharing the knowledge about semi-algebraic sets, to Thomas Lam, whose great observations increased the generality of the results in this paper, and to Steven Karp for helpful comments.
We are grateful to the referee for careful reading of our manuscript and for useful suggestions that improved the quality of our paper.


\section{Proof of Theorem \ref{thm:fibers}}
\label{sec:contractible}

Let $k \geq 1, m \ge0$ and $n \geq k+m$ be integers and let $Z$ be a real $(k+m)\times n$ matrix such that the map $\widetilde{Z}$ is well defined. 
Since the action of the group $\GL_k^+$ on $\Mat_{k,n}^{\geq 0}$ is free, there is a fibration
\begin{equation}
\label{eqn:fibration nonnegative Grassmannian}
\GL_k^+ \longrightarrow \Mat_{k,n}^{\ge0} \longrightarrow \gr_k^{\ge0}(\R^n).
\end{equation}

The matrix $Z$, as in Definition \ref{def:amplituhedron}, induces a map 
\begin{eqnarray*}
\widehat{Z} : \Mat_{k,n}^{\ge0} & \longrightarrow & \Mat_{k,k+m},\\
V & \longmapsto & VZ^\top,
\end{eqnarray*}
which is again well defined, see for example \cite[Prop.\,15.2]{Lam2016}. 

Let $e$ be a positroid cell in the CW decomposition of $\nng{k}{n}$, and let $I_e \subseteq \binom{[n]}{k}$ be the family of nonbases (dependent sets) of cardinality $k$ of the matroid that defines the cell $e$. 
The maximal minors of a $k \times n$ matrix are indexed by the set $\binom{[n]}{k}$. Denote by $\Mat_{k,n}^{\ge0}(e)$ the set of all matrices $V \in \Mat_{k,n}^{\ge0}$ whose minors indexed by elements of $I_e$ are equal to zero. 
Then every point in $\bar{e} \subseteq \nng{k}{n}$ is represented by a matrix in $\Mat_{k,n}^{\ge0}(e)$, and the row span of every such matrix lies in $\bar{e}$.
In other words, $\bar{e}= \GL_k^{+} \backslash \Mat_{k,n}^{\geq 0}(e)$.
Thus the restriction of the fibration \eqref{eqn:fibration nonnegative Grassmannian} is a fibration 
\begin{equation}
\label{eqn:fibration positroid cell}
\GL_k^+ \longrightarrow \Mat_{k,n}^{\ge0}(e) \longrightarrow \bar{e}.
\end{equation}
Note that if $e$ is the maximal positroid cell, the set $\Mat_{k,n}^{\ge0}(e)$ is the whole set $\Mat_{k,n}^{\ge0}$.

Denote by $\widehat{\GP{e}}$ the image of the set $\Mat_{k,n}^{\ge0}(e)$ under the map $\widehat{Z}$. With a usual abuse of notation, we consider maps $\widehat{Z}: \Mat_{k,n}^{\ge0}(e) \longrightarrow \widehat{\GP{e}}$ and $\widetilde{Z}:\bar{e} \longrightarrow \GP{e}$. Then there exists a commutative diagram of spaces and continuous maps
\[
\begin{tikzcd}  
 \Mat_{k,n}^{\ge0}(e)   \arrow{r}{\widehat{Z}} \arrow{d}  & \widehat{\GP{e}} \arrow{d}      \\
\bar{e}	  \arrow{r}{\widetilde{Z}}  & \GP{e},
\end{tikzcd}
\]
where vertical maps send any matrix to its row span.

\medskip

The proof of Theorem \ref{thm:fibers} splits into the following two lemmas.

\begin{lemma}
\label{lemma:fiber}
Let $k \geq 1$ and $m\ge0$ be integers, $n=k+m+1$, and let $Z$ be a real $(k+m)\times n$ matrix such that the map $\widetilde{Z}$ is well defined.
Then for every positroid cell $e$ in the CW decomposition of $\nng{k}{n}$ and for every $W \in \widehat{\GP{e}}$, the inverse image $\widehat{Z}^{-1}(\{W\}) \subseteq \Mat_{k,n}^{\ge0}(e)$ is nonempty and convex.
\end{lemma}

\begin{proof}
The matrix $Z$ induces a linear map 
\begin{eqnarray}
\label{eqn:map corank 1}
\R^n &\longrightarrow& \R^{k+m}\\
v &\longmapsto& vZ^\top \nonumber,
\end{eqnarray}
where $v \in \R^n$ is a row vector. Since $n=k+m+1$, the kernel of the map \eqref{eqn:map corank 1} is $1$-dimensional. Fix a generator $a \in \R^n$ of that kernel.

Choose an arbitrary point $W \in \widehat{\GP{e}}$, and let $U$ and $V$ be any two points in $\widehat{Z}^{-1}(\{W\})$. 
Our goal is to show that for every $\lambda \in [0,1]$ the convex combination $(1-\lambda)U+\lambda V$ also belongs to $\widehat{Z}^{-1}(\{W\})$. 

Since $UZ^\top=VZ^\top=W$, the rows of the matrix $V-U$ belong to $\ker(Z)$.
Consequently, there exists a row vector $x \in \R^k$ such that $V-U = x^\top a$, where $a$ is also considered as a row vector.
Thus we have to show that for every $\lambda \in [0,1]$ the convex combination 
\begin{equation}
\label{eqn:convex combination}
(1-\lambda)U + \lambda V = U + \lambda x^\top a
\end{equation}
belongs to the space $\Mat_{k,n}^{\ge0}(e)$, this means that every $k \times k$ minor of the matrix \eqref{eqn:convex combination} is nonnegative, and in addition that all the minors of the matrix \eqref{eqn:convex combination} indexed by the nonbases $I_e \subseteq \binom{[n]}{k}$ of the matroid corresponding to $e$ are equal to zero.

A $k \times k$ submatrix of the matrix \eqref{eqn:convex combination} is of the form 
\begin{equation}
\label{eqn:matrix minor}
\left(
\begin{matrix}
u_{1i_1}+\lambda x_1a_{i_1} & \dots & u_{1i_k}+\lambda x_1a_{i_k} \\
\vdots & & \vdots \\
u_{ki_1}+\lambda x_ka_{i_1} & \dots & u_{ki_k}+\lambda x_ka_{i_k}
\end{matrix}
\right),
\end{equation}
where 
\[
U=\left(
\begin{matrix}
u_{11} & \dots & u_{1n} \\
\vdots & &\vdots \\
u_{k1} & \dots & u_{kn}
\end{matrix}
\right), \
x=(x_1 \dots x_k), \  a=(a_1 \dots a_n),
\]
and $1 \leq i_1 < \dots <i_k \leq n$. The matrix \eqref{eqn:matrix minor} can be transformed using row operations into a matrix that contains the variable $\lambda$ only in one row. Therefore, every $k\times k$ minor of the matrix \eqref{eqn:convex combination} is a polynomial of degree at most $1$ in the variable $\lambda$. Since it takes nonnegative values for $\lambda=0$ and $\lambda=1$, it is also nonnegative for all $\lambda\in [0,1]$. 
Thus for every $\lambda\in [0,1]$, the point $(1-\lambda)U+\lambda V$ belongs to $\Mat_{k,n}^{\ge0}$. 
Similarly, if $\{i_1, \dots, i_k\}$ is a nonbasis of the matroid corresponding to $e$, then the determinant of the matrix \eqref{eqn:matrix minor} is zero for $\lambda=0$ and $\lambda=1$, so it is a constant zero-polynomial, meaning that the matrix \eqref{eqn:convex combination} belongs to $\Mat_{k,n}^{\ge0}(e)$ for every $\lambda \in [0,1]$. Consequently the set $\widehat{Z}^{-1}(\{W\})$ is convex. 
\end{proof}

\begin{lemma}
\label{lemma:homeomorphism}
Let $k\geq 1, m\ge0$ and $n\geq k+m$ be integers. 
For every positroid cell $e$ and for every $W \in \widehat{\GP{e}}$, the inverse images 
\[
\widehat{Z}^{-1}(\{W\}) \subseteq \Mat_{k,n}^{\ge0}(e) \subseteq \Mat_{k,n}^{\ge0}
\qquad\text{and}\qquad 
\widetilde{Z}^{-1}(\{\span(W)\}) \subseteq \bar{e} \subseteq \gr_k^{\ge0}(\R^n)
\] 
are homeomorphic.
\end{lemma}

\begin{proof}
Let $\varphi \colon \widehat{Z}^{-1}(\{W\}) \longrightarrow \widetilde{Z}^{-1}(\{\span(W)\})$ be defined by $\varphi(U)=\span(U)$, where $U\in \widehat{Z}^{-1}(\{W\})$, and $\span$ denotes the row span. 
We prove that $\varphi$ is a homeomorphism.

Clearly, $\varphi$ is continuous, so it suffices to find a continuous map $\psi:\widetilde{Z}^{-1}(\{\span(W)\}) \longrightarrow \widehat{Z}^{-1}(\{W\})$ such that $\varphi \circ \psi$ is the identity map on $\widetilde{Z}^{-1}(\{\span(W)\})$ and $\psi \circ \varphi$ is the identity map on $\widehat{Z}^{-1}(\{W\})$. Let $L \in \widetilde{Z}^{-1}(\{\span(W)\})$. Then there exists a matrix $K \in \Mat_{k,n}^{\geq0}(e)$ whose rows span the subspace $L$. Since
\[
\span(KZ^\top) = \span(W),
\]
there exists a unique $C \in \GL_k$ such that $KZ^\top=CW$. Now define $\psi$ as $\psi(L)=C^{-1}K$. It can be seen using Cauchy--Binet formula that $\det(C)>0$, thus, $C^{-1}K \in \Mat_{k,n}^{\ge0}(e)$.
Even though we have defined the map $\psi$ using an arbitrarily chosen matrix $K$ such that $\span(K)=L$, it can be checked directly that the definition of $\psi$ does not depend on a choice of $K$.

In order to prove that the map $\psi$ is continuous, we need to show that the choice of a matrix $K$ can be made continuously on $\widetilde{Z}^{-1}(\{\span(W)\})$. 
The choice of a matrix $K$ is equivalent to the choice of a positively oriented basis for the subspace $L \subseteq \R^n$. 
Therefore, we need a continuous section of the fiber bundle \eqref{eqn:fibration positroid cell} restricted to the set $\widetilde{Z}^{-1}(\{\span(W)\})$. 
Since the base space $\bar{e}$ is contractible, the fiber bundle \eqref{eqn:fibration positroid cell} is trivial. 
In particular, its restriction on $\widetilde{Z}^{-1}(\{\span(W)\})$ is also trivial, so it admits a continuous section. 
Therefore, the bases for elements of $\widetilde{Z}^{-1}(\{\span(W)\})$  can be chosen continuously. 
On the other hand, the matrix $C$ is a solution of the linear system $KZ^\top=CW$, which depends continuously on $K$, thus it also depends continuously on $L$. 

Lastly, 
\[
\varphi(\psi(L))=\varphi(C^{-1}K)=\span(C^{-1}K)=\span(K)=L,
\]
holds for every $L \in \widetilde{Z}^{-1}(\{\span(W)\})$, and
\[
\psi(\varphi(U))=\psi(\span(U))=C^{-1}U,
\]
for every $U \in \widehat{Z}^{-1}(\{W\})$, where $C$ is the unique $k \times k$ matrix such that $W=\widehat{Z}(U)=UZ^\top=CW$, hence $C$ is the identity matrix.
\end{proof}

Finally, Lemma \ref{lemma:fiber} and Lemma \ref{lemma:homeomorphism} complete the proof of Theorem \ref{thm:fibers}.

\section{Proof of Theorem \ref{thm:triangulation}}
\label{sec:triangulation}

Let us fix an arbitrary positroid cell $e$ in the CW decomposition of the totally nonnegative Grassmannian $\nng{k}{n}$. 

Furthermore, let $d=\binom{k+m}{k}$, and consider the Veronese embedding $\nu \colon \RP^{d-1} \longrightarrow \R^{d\times d}$ given by
\[
x=(x_1: \ldots :x_d) \longmapsto \left( \frac{x_i x_j}{x_1^2+\dots +x_d^2} \right)_{1\leq i,j\leq d},
\]
where $x=(x_1: \ldots :x_d)\in \RP^{d-1}$.
The embedding $\nu$ maps every line $x \in \RP^{d-1}$ to the matrix of the projection $\R^d \longrightarrow x$. 
For more details on the Veronese embedding see for example \cite[Sec.\,3.4.2]{BochnakEtAl1998}.

Consider next, with obvious abuse of notation,  the continuous map $\nu\colon \R^d{\setminus}\{0\} \longrightarrow \R^{d\times d}$ given by
\[
(x_1, \dots, x_d) \longmapsto \left( \frac{x_i x_j}{x_1^2+\dots +x_d^2} \right)_{1\leq i,j\leq d} \in \R^{d \times d}.
\]

In this way, we obtain the commutative diagram of spaces and maps
\[
\begin{tikzcd}  
 \Mat_{k,n}^{\ge0}   \arrow{r}{\widehat{Z}} \arrow{d} & \Mat_{k,k+m} \arrow{r}{\gamma} \arrow{d} & \R^d \setminus \{0\} \arrow{d}{\pi}  \arrow{r}{\nu} & \R^{d\times d} \arrow{d}{\id}     \\
\nng{k}{n} \arrow{r}{\widetilde{Z}} & \gr_k(\R^{k+m}) \arrow{r}{\gamma} & \RP^{d-1} \arrow{r}{\nu} & \R^{d\times d} ,
\end{tikzcd}
\]
where $\gamma\colon \gr_k(\R^{k+m}) \longrightarrow \RP^{d-1}$ is the Pl\"ucker embedding, $\gamma\colon  \Mat_{k,k+m} \longrightarrow \R^{d}\setminus \{0\}$ maps every matrix to the tuple of its $k \times k$ minors, and $\pi\colon \R^{d}\setminus \{0\} \longrightarrow \RP^{d-1}$ is the quotient map.

The Grassmann polytope $\GP{e}=\widetilde{Z}(\bar{e})$ is embedded into $\RP^{d-1}$ via $\gamma$, the projective space $\RP^{d-1}$ is embedded into the Euclidean space $\R^{d \times d}$ via $\nu$, and thus the image $\nu(\gamma(\GP{e}))$ is homeomorphic to $\GP{e}$.

First, we prove that the homeomorphic image of the Grassmann polytope $\nu(\gamma(\GP{e}))$ is semi-algebraic. 
The commutativity of the diagram above implies that
\[
\nu(\gamma(\GP{e}))  = \nu(\pi(\gamma(\widehat{\GP{e}}))) = \nu(\gamma(\widehat{\GP{e}}))= \nu(\gamma(\widehat{Z}(\Mat_{k,n}^{\geq 0}(e)))).
\]
The set $\Mat_{k,n}^{\ge0}(e) \subseteq \R^{k \times n}$ is semi-algebraic, even algebraic. 
Since the map $\widehat{Z}$ is multiplication by a matrix, the set $\widehat{\GP{e}}$ is also semi-algebraic \cite[Cor.\,2.4(2)]{Coste2002}.
Furthermore, the map $\gamma\colon \Mat_{k,k+m} \longrightarrow \R^{d}\setminus \{0\}$ is a restriction of a polynomial map $\R^{k\times (k+m)}\longrightarrow\R^d$, and thus $\gamma(\widehat{\GP{e}}) \subseteq \R^{d} \setminus \{0\}$ is semi-algebraic by \cite[Cor.\,2.4(2)]{Coste2002} as well. 
Finally, the map $\nu\colon\R^d \longrightarrow \R^{d \times d}$ is a regular rational map, and consequently it maps semi-algebraic sets to the semi-algebraic sets, see \cite[Prop.\,2.2.7]{BochnakEtAl1998} \cite[Cor.\,2.9(1)]{Coste2002}.
Hence, we have proved that $\nu(\gamma(\GP{e}))$ is semi-algebraic in $\R^{d\times d}$, and consequently the Grassmann polytope $\GP{e}$ is homeomorphic to a semi-algebraic set.
In particular, since $\GP{e}$ is compact and homeomorphic to a semi-algebraic set it admits a triangulation according to \cite[Thm.\,9.2.1]{BochnakEtAl1998} \cite[Thm.\,3.11]{Coste2002}.

\medskip
Second, notice that we obtained a bit more.
The $\pi$ inverse image of the embedded Grassmann polytope $\GP{e}$ via $\gamma$ can be presented as follows 
\[
\pi^{-1}(\gamma(\GP{e}))=\gamma(\widehat{Z}(\Mat_{k,n}^{\ge0}(e)))\cup \big(-\gamma(\widehat{Z}(\Mat_{k,n}^{\geq 0}(e)))\big).
\]
Since we proved that $\gamma(\widehat{Z}(\Mat_{k,n}^{\geq 0}(e)))$ is semi-algebraic we can conclude that $\pi^{-1}(\gamma(\GP{e}))$ is also a semi-algebraic subset of $\R^d$.

Having in mind that every real projective variety is affine,  we can define that a subset $X$ of the real projective space $\RP^{d-1}$ is semi-algebraic if, for example, its  preimage $\pi^{-1}(X)\subseteq\R^d$, via the defining quotient map $\R^d\longrightarrow \RP^{d-1}$, is semi-algebraic.

Thus, we proved that the Grassmann polytope $\GP{e}$, when embedded in $\RP^{d-1}$ via the Pl\"ucker embedding, is a semi-algebraic subset of the real projective space.

\qed

\section{Proof of Theorem \ref{thm:homeomorphic}}
\label{sec:homeomorphic}

Let $ k \geq 1$, $m \ge0$ and $n \geq k+m$ be integers, and suppose in addition that $m$ is even. Let $S \in \GL_n$ be given by 
\[
S(x_1, \dots, x_n) = (x_2, \dots, x_n, (-1)^{k-1}x_1).
\]
Denote by $Z_0 \in \Mat_{k+m,n}$ the matrix whose rows are the eigenvectors of the matrix $S+S^\top$ that correspond to the largest $k+m$ eigenvalues. 
It was shown in \cite[Lemma~3.1]{Galashin2017} that all $(k+m)\times (k+m)$ minors of the matrix $Z_0$ are positive, thus it defines an amplituhedron $\A_{n,k,m}(Z_0)$, called \emph{cyclically symmetric amplituhedron}.
Galashin, Karp \& Lam \cite[Thm.\,1.2]{Galashin2017} showed that $\A_{n,k,m}(Z_0)$ is homeomorphic to a closed $km$-dimensional ball whenever the parameter $m$ is even.

We conclude the proof of Theorem \ref{thm:homeomorphic} by showing that the amplituhedra $\A_{n,k,m}(Z)$ and $\A_{n,k,m}(Z_0)$ are homeomorphic.

From \cite[Cor.\,1.12(ii)]{Karp2017b} we know that entries of every nonzero vector of $\ker(Z_0)$ and of $\ker(Z)$ are nonzero, and they alternate in sign. 
Since $n=k+m+1$, the kernels of matrices $Z$ and $Z_0$ are $1$-dimensional. 
Let $a=(a_1, \dots, a_n)\in \R^n$ be a generator of the kernel of $Z$ and let $b=(b_1, \dots, b_n)\in \R^n$ be  a generator of the kernel of $Z_0$ (it follows from the cyclic symmetry of $Z_0$ that $b_i=(-1)^{i-1}$ for $1\leq i\leq n$, see~\cite{Galashin2017}). 
Choose them in such a way that $a_1$ and $b_1$ have the same sign. 
Consequently, for every $1 \leq i \leq n$, the entries $a_i$ and $b_i$ have the same sign. 
Let $D$ be an $n \times n$ diagonal matrix $D=\diag(\frac{a_1}{b_1}, \dots, \frac{a_n}{b_n})$. 
The matrix $ZD$ has the same kernel as the matrix $Z_0$, and since the diagonal entries of the matrix $D$ are positive, all maximal minors of the matrix $ZD$ are positive. 
The fact that the matrices $ZD$ and $Z_0$ have the same kernel implies that they have the same row spans, as well. 
In particular, there exists a matrix $C \in \GL_{k+m}^+$ such that $Z_0  = CZD$.

Multiplication by $D$ on the right gives a homeomorphism $\widehat{D}: \Mat_{k,n}^{\ge0} \longrightarrow\Mat_{k,n}^{\ge0}$, which induces a homeomorphism $\widetilde{D}:\nng{k}{n} \longrightarrow \nng{k}{n}$. Furthermore, multiplication by $C^\top$ on the right gives a homeomorphism $\widehat{C}:\Mat_{k,k+m} \longrightarrow \Mat_{k,k+m}$, thus the induced map $\widetilde{C}: \gr_k(\R^{k+m}) \longrightarrow \gr_k(\R^{k+m})$ is also a homeomorphism. Hence, we obtain the commutative diagram of spaces and maps
\[
\begin{tikzcd}  
 \Mat_{k,n}^{\ge0}   \arrow{r}{\widehat{D}} \arrow{d} & \Mat_{k,n}^{\ge0}   \arrow{r}{\widehat{Z}} \arrow{d}  &\Mat_{k,k+m} \arrow{r}{\widehat{C}} \arrow{d} & \Mat_{k,k+m} \arrow{d}      \\
\nng{k}{n} \arrow{r}{\widetilde{D}} & \nng{k}{n} \arrow{r}{\widetilde{Z}}  & \gr_k(\R^{k+m}) \arrow{r}{\widetilde{C}} & \gr_k(\R^{k+m}).
\end{tikzcd}
\]
The image of the composition $\widetilde{C}\circ\widetilde{Z}\circ\widetilde{D}$ of the maps in the lower row of the diagram is the cyclically symmetric amplituhedron $\A_{n,k,m}(Z_0)$ and the image of the map $\widetilde{Z}$ is the amplituhedron $\A_{n,k,m}(Z)$. Since the maps $\widetilde{C}$ and $\widetilde{D}$ are homeomorphisms, these two amplituhedra are homeomorphic. Finally, the fact that the cyclically symmetric amplituhedron $\A_{n,k,m}(Z_0)$ is homeomorphic to a $km$-dimensional ball \cite[Thm.\,1.2]{Galashin2017}, when $m$ is even, concludes the argument that every amplituhedron $\A_{n,k,m}(Z)$ is homeomorphic to a $km$-dimensional ball whenever $n=k+m+1$ and $m$ is even.\qed

\medskip
The proof of Theorem \ref{thm:homeomorphic} gives even more.
Let us say that two Grassmann polytopes $\GPZ{e}{Z},\GPZ{e'}{Z'}\subseteq \gr_k(\R^{k+m})$ are \emph{projectively equivalent} if there exists a matrix $M\in \GL_{k+m}$ such that
\[
\GPZ{e'}{Z'}=\{\widetilde{M}(x)\mid x\in \GPZ{e}{Z}\}.
\] 
Here $\widetilde{M}$ denotes a map $\gr_k(\R^{k+m})\longrightarrow \gr_k(\R^{k+m})$ induced by the natural action of $M$ on $\R^{k+m}$.
For $k=1$, this coincides with the standard notion of projective equivalence for polytopes in the projective space $\RP^{k+m-1}$. 
The proof of Theorem \ref{thm:homeomorphic} actually shows that for $n=k+m+1$, any two amplituhedra are projectively equivalent.


\end{document}